\documentclass[11pt]{amsart}

\usepackage{amsmath}
\usepackage{fullpage}
\usepackage{xspace}
\usepackage[psamsfonts]{amssymb}
\usepackage[latin1]{inputenc}
\usepackage{graphicx,color}
\usepackage[curve]{xypic} 
\usepackage{hyperref}
\usepackage{graphicx}

\usepackage{amsmath}%
\usepackage{amsthm}%
\usepackage{amscd}
\usepackage{amsfonts}%
\usepackage{amssymb}%
\usepackage{graphicx}

\usepackage{mathrsfs}

\usepackage{tikz}
\usetikzlibrary{matrix,arrows}

\usepackage{tikz-cd}

\newtheorem{theorem}{Theorem}[section]

\newtheorem{corollary}[theorem]{Corollary}

\theoremstyle{remark}

\numberwithin{equation}{section}

\newcommand{\Z}{\mathbb{Z}}

\newcommand{\Q}{\mathbb{Q}}

\newcommand{\rad}{\mathrm{rad}}

\makeatletter
\@namedef{subjclassname@2020}{%
  \textup{2020} Mathematics Subject Classification}
\makeatother

  \DeclareFontFamily{U}{wncy}{}
    \DeclareFontShape{U}{wncy}{m}{n}{<->wncyr10}{}
    \DeclareSymbolFont{mcy}{U}{wncy}{m}{n}
    \DeclareMathSymbol{\Sha}{\mathord}{mcy}{"58}

\begin{document}
\title[]{An improvement on the largest prime factor of $n^2+1$}

\author{Hector Pasten}
\address{ Departamento de Matem\'aticas,
Pontificia Universidad Cat\'olica de Chile.
Facultad de Matem\'aticas,
4860 Av.\ Vicu\~na Mackenna,
Macul, RM, Chile}
\email[H. Pasten]{hector.pasten@uc.cl}%


\thanks{H.P. was supported by ANID Fondecyt Regular grant 1230507 from Chile.}
\date{\today}

\begin{abstract} The study of the largest prime factor in polynomial sequences can be traced back at least to the late 19th century in the work of St\"ormer.  Mahler (1933) and Chowla (1934) proved that the largest prime factor of $n^2+1$ grows at least as fast as $\log_2 n$. In 2023 we improved this to $(\log_2 n)^2/\log_3 n$. In this note we show the lower bound $(\log_2 n)^2/\log_4 n$, and that when this bound is nearly sharp it also holds for many prime factors of $n^2+1$. 
\end{abstract}

\maketitle



\section{Introduction}


The study of the prime factorization of integers of the form $n^2+1$ has a long history. In the late 19th century, St\"ormer \cite{Stormer1, Stormer2} investigated the growth of the largest prime factor of quadratic sequences, in particular, of $n^2+1$. Let us write $P(m)$ for the largest prime factor of a positive integer $m$ and $\log_k$ for the $k$-th iterated logarithm (whenever it is defined). In the early 1930's, Chowla \cite{Chowla} and Mahler \cite{Mahler} independently proved
$$
P(n^2+1) \gg \log_2 n
$$
as $n$ grows. This bound remained the best result until the work of Haristoy \cite{Haristoy} (after techniques of Stewart--Yu \cite{StewartYu} based on linear forms in logarithms) where the estimate
$$
P(n^2+1) \gg \frac{\log_3 n}{\log_4 n}\log_2 n
$$
was established. Note that this still has the form $(\log_2 n)^{1+o(1)}$. In 2023 we used the theory of Shimura curves \cite{PastenShimuraCurves} and linear forms in logarithms to obtain the substantial improvement \cite{Pasten}:
\begin{equation}\label{EqnPasten}
P(n^2+1) \gg \frac{(\log_2 n)^2}{\log_3 n}.
\end{equation}
Our main result is a ``hybrid'' bound from which one can easily deduce an improvement of the previous bound. For this, let $\rad(m)$ be the radical of $m\ne 0$, that is, the largest squarefree divisor of $m$. For a positive integer $n$, let $P_n = P(n^2+1)$ and $R_n=\rad(n^2+1)$.

\begin{theorem}\label{ThmMain}  As $n$ grows, we have
$$
(\log_2 n )^2 \ll \log(R_n)\log_2(P_n). 
$$
\end{theorem}

From here we obtain:

\begin{corollary}\label{CoroP} As $n$ grows, we have
$$
P_n \gg \frac{(\log_2 n)^2}{\log_4 n}.
$$
\end{corollary}
This result improves \eqref{EqnPasten} by a factor of $(\log_3 n)/\log_4n$.

To deduce Corollary \ref{CoroP} from Theorem \ref{ThmMain} one first notes that Chebyshev's bound for primes gives
$$
\log R_n \le \sum_{p\le P_n} \log p \ll P_n
$$
from where one deduces
$$
(\log_2 n)^2 \ll P_n\log_2(P_n),
$$
and finally one inverts to extract $P_n$.

One might ask what would happen in the hypothetical scenario in which the bound of Corollary \ref{CoroP} is (nearly) sharp. It turns out that in that case the same bound would hold for many prime factors of $n^2+1$.
\begin{corollary}\label{CoroMany} There is an absolute constant $\kappa>0$ such that for all $n\gg 1$ we have
$$
\#\left\{p\, |\, n^2+1 : p > \kappa \frac{(\log_2 n)^2}{\log_2 P_n}\right\} \gg \frac{(\log_2 n)^2}{(\log P_n) \log_2 P_n}.
$$ 
In particular, if an integer $n$ satisfies $P_n \le A (\log_2 n)^B$ for fixed constants $A>0$ and $B\ge 2$, then there are at least
$$
\gg_{A,B} \frac{(\log_2 n)^2}{(\log_3 n)\log_4 n}
$$
prime divisors $p$ of $n^2+1$ satisfying
$$
p \gg_{A,B} \frac{(\log_2 n)^2}{\log_4 n}.
$$
\end{corollary}
The deduction of this corollary from Theorem \ref{ThmMain} is presented in Section \ref{SecProofs}.

The results in this note are discussed for the polynomial $n^2+1$ for historical reasons and because in this case the arguments are cleaner than for other polynomials (essentially, because $\Z[i]$ is a UFD). But the main ideas of Section \ref{SecProofs} can be reused in the context of the arguments of \cite{CuevasPasten} to get similar results for all irreducible quadratic polynomials and for cubic polynomials of the form $ax^3+b$. We do not pursue that direction here because the adaptation, after the methods of \cite{CuevasPasten} and the ideas presented here, is mostly routine.

The use of AI is described in the Acknowledgements section.


\section{Proofs}\label{SecProofs}

\begin{proof}[Proof of Theorem \ref{ThmMain}]
The proof follows the ideas from \cite{Pasten}.

Factor
$$
 n+i=u\prod_{j=1}^{r}\gamma_j^{e_j}
$$
in $\mathbb Z[i]$, with the $\gamma_j$ pairwise non-associated Gaussian
primes. For $B>2$, let
$$
 I=\{j:e_j>B\},\qquad m=1+\#I,
$$
and let $p_j$ be the rational prime below $\gamma_j$. 

Let $\gamma_0$ be the part of the factorization of $n+i$ coming from the primes $\gamma_j$ with $j\notin I$. Let $w=\bar{u}/u$ and $\xi_j = \bar{\gamma}_j/\gamma_j$ where the bar denotes complex conjugation. Then 
$$
\xi := \frac{n-i}{n+i} = w\xi_0\prod_{j\in I} \xi_j^{e_j}
$$
belongs to the multiplicative subgroup of $\Q(i)^\times$ generated by the $\xi_j$ modulo torsion.

We observe that
$$
1-\xi = 1-\frac{n-i}{n+i} = \frac{2i}{n+i}.
$$
From Theorem 4.2.1 in \cite{EvertseGyory} (a consequence of the theory of linear forms in logarithms) we obtain:
$$
\log\frac{|n+i|}{2} = - \log\left|1-\xi\right| \le C^m \left(\log\max\{2 , h(\xi)\}\right) \max\{1,h(\xi_0)\}\prod_{j\in I} h(\xi_j)
$$
for an absolute constant $C>1$. Since $h(\xi) \ll \log n$, and we also have $h(\xi_0)\ll B\log R_n$ and $h(\xi_j)\ll \log p_j$, we deduce
$$
\sqrt{\log n}\le \frac{\log n}{\log_2 n} \le K^mB(\log R_n)\prod_{j\in I} \log p_j \le B(\log R_n)(K \log P_n)^m
$$
for an absolute constant $K>1$ and $n\gg 1$. We will bound $m$ in terms of the choice of $B$.

Lemma 3.2 from \cite{Pasten} (which comes from the theory of Shimura curves of \cite{PastenShimuraCurves}) gives
$$
(B/2)^{m-1}\le  \prod_{p|n^2+1} v_p(n^2+1) \ll R_n^8
$$
where $v_p(N)$ is the exponent of $p$ in the prime factorization of $N$. (The bound has $B/2$ instead of $B$ because the exponents $e_j$ come from the factorization of $n+i$ in $\Z[i]$, not $n^2+1$ in $\Z$.)

Hence
$$
m \ll 1+ \frac{\log R_n}{\log B}.
$$
Choosing $B = \exp(\sqrt{(\log R_n)\log_2 P_n})$ we get
$$
m \ll  \sqrt{(\log R_n)/\log_2 P_n}
$$
hence
$$
\log_2 n \ll \log B + \log_2 R_n + m\log_2 P_n \ll \sqrt{(\log R_n)\log_2 P_n} + \log_2 R_n + \sqrt{(\log R_n)\log_2 P_n}
$$
and the result follows.
\end{proof}


\begin{proof}[Proof of Corollary \ref{CoroMany}] Let $T$ be a parameter to be chosen later. Let $D_n(T)$ be the number of primes $p>T$ that divide $n^2+1$. Then
$$
\log R_n \le \sum_{p\le T} \log p + \sum_{\substack{p>T\\ p|n^2+1}} \log p \le 2 T + D_n(T)\log P_n
$$
by Chebyshev's bound. Theorem \ref{ThmMain} gives
$$
\log R_n \ge c\cdot \frac{(\log_2 n)^2}{\log_2P_n}
$$
as $n$ grows, for a suitable absolute constant $c>0$. Hence
$$
2 T + D_n(T)\log P_n \ge c\cdot\frac{(\log_2 n)^2}{\log_2P_n}
$$
and we deduce
$$
D_n(T) \ge c\cdot \frac{(\log_2 n)^2}{(\log P_n)\log_2P_n} - \frac{2T}{\log P_n}.
$$
The result follows by choosing 
$$
T = (c/3) \frac{(\log_2 n)^2}{\log_2P_n}
$$
provided that $n\gg 1$.
\end{proof}

\section{Acknowledgements}

Supported by ANID Fondecyt Regular grant 1230507 from Chile. 

Corollary \ref{CoroP} and its initial proof were first obtained in an autonomous way by the AI model ChatGPT-5.6 Sol Pro after a single prompt in July 2026. Theorem \ref{ThmMain} was subsequently discovered by the author from interactions with the same AI model. The proof presented here is due to the author. Claude Opus 5 was used for proofreading an earlier version of the manuscript. The author takes full responsibility for the content of the paper.



\begin{thebibliography}{9}         

\bibitem{Chowla} S. Chowla, \emph{The greatest prime factor of $x^2+1$}. J. London Math. Soc. 10 (1935), 117--120.

\bibitem{CuevasPasten} J. Cuevas Barrientos, H. Pasten, \emph{On the greatest prime factor of polynomial values and subexponential Szpiro in families}. Preprint (2025) arXiv:2504.15971


\bibitem{EvertseGyory} J.-H. Evertse, K. Gy\"ory, \emph{Unit Equations in Diophantine Number Theory}. Cambridge Studies in Advanced Mathematics, Cambridge University Press, 2016. 

\bibitem{Haristoy} J. Haristoy, \emph{\'Equations diophantiennes exponentielles}. PhD, University of Strasbourg I (Louis Pasteur), 2003.

\bibitem{Mahler} K. Mahler, \emph{\"Uber den gr\"ossten Primteiler der Polynome $x^2\mp 1$}. Archiv Math. og Naturv. 41 (1933), 3--8.

\bibitem{PastenShimuraCurves} H. Pasten, \emph{Shimura curves and the abc conjecture}. J. Number Theory 254 (2024), 214--335.

\bibitem{Pasten} H. Pasten, \emph{The largest prime factor of $n^2+1$ and improvements on subexponential $ABC$}. Invent. Math. 236 (2024), 373--385.

\bibitem{StewartYu} C. L. Stewart, K. Yu, \emph{On the abc conjecture II}. Duke Math. J. 108 (2001), 169--181.

\bibitem{Stormer1} C. St\"ormer, \emph{Quelques th\'eor\`emes sur l'\'equation de Pell $x^2 - Dy^2 = \pm 1$ et leurs applications}. Skrifter Videnskabsselskabet (Christiania), Mat.-Naturv. Kl. I (1897).

\bibitem{Stormer2} C. St\"ormer, \emph{Sur une  \'equation indetermin\'ee}. C. R. Acad. Sci. Paris 127 (1898), 752--754.

\end{thebibliography}
\end{document}